\documentclass[11pt]{amsart}
\usepackage{amssymb,verbatim,graphicx,amsmath,mathtools}
\usepackage{graphics,epsfig,psfrag} 
\usepackage[inline]{enumitem}
\usepackage{pb-diagram,hyperref,graphics,epsfig,psfrag}
\usepackage{nameref}
\usepackage{caption}
\usepackage{subcaption}
\hypersetup{colorlinks}
\usepackage{ulem,color,braket}
\definecolor{darkred}{rgb}{0.5,0,0}
\definecolor{darkgreen}{rgb}{0,0.5,0}
\definecolor{darkblue}{rgb}{0,0,0.5}
\hypersetup{ colorlinks,
linkcolor=darkblue,
filecolor=darkgreen,
urlcolor=darkred,
citecolor=darkblue }
\newtheorem{theorem}{Theorem}%[section]

\newtheorem{assumption}[theorem]{Assumption}%[subsection]

\newtheorem{proposition}[theorem]{Proposition}
\newtheorem{lemma}[theorem]{Lemma}

\theoremstyle{definition}
\newtheorem{definition}[theorem]{Definition}
\theoremstyle{remark}
\newtheorem{remark}[theorem]{Remark}
\newtheorem{example}[theorem]{Example}

\newtheorem*{claim}{Claim}
\newtheorem*{observation}{Observation}

\renewcommand\em{\it}
 \newcommand{\J}{\mathcal{J}}

\newcommand{\R}{\mathbb{R}}

\newcommand{\C}{\mathbb{C}}
\newcommand{\Z}{\mathbb{Z}}

\newcommand{\A}{\mathcal{A}}
\newcommand{\D}{\mathcal{D}}
\newcommand{\X}{X}

\newcommand{\on}{\operatorname}
 \newcommand\bs{\backslash}

\newcommand{\hra}{\hookrightarrow}
 
\renewcommand{\d}{{\on{d}}}
\newcommand{\ol}{\overline}

\newcommand\Phinv{\Phi^{-1}}

\newcommand\lam{\lambda}
\newcommand\Lam{\Lambda}

\newcommand\sig{\sigma}
\newcommand\eps{\epsilon}

\newcommand\Om{\Omega}
\newcommand\om{\omega}
\newcommand{\lan}{\langle}
\newcommand{\ran}{\rangle}
\newcommand{\hh}{{\frac{1}{2}}}
\newcommand{\thh}{{\tfrac{1}{2}}}
\newcommand{\ddr}{{\tfrac{\d}{\d r}}}
\newcommand{\dds}{{\tfrac{\d}{\d s}}}
\newcommand{\ddrho}{{\tfrac{\d}{\d \rho}}}

\newcommand{\pps}{{\frac{\partial}{\partial s}}}
\newcommand{\pptau}{{\frac{\partial}{\partial \tau}}}
\newcommand{\pprho}{{\frac{\partial}{\partial \rho}}}

\newcommand\cT{\mathcal{T}}
\newcommand\Vect{\on{Vect}}
\newcommand\bran[1]{ \lan {#1} \ran}
\newcommand\Id{\on{Id}}
\newcommand{\cyl}{{\on{cyl}}}
\newcommand{\Hof}{{\on{Hof}}}
\newcommand{\hor}{{\on{hor}}}
\newcommand{\loc}{{\on{loc}}}

\newcommand\cpt{{\on{cpt}}}
\newcommand\symp{{\on{cpt}}}
\newcommand\seps{\epsilon_{\on{cpt}}}
\newcommand\tw{{\on{tw}}}
\newcommand\pre{{\on{pre}}}
\newcommand\monot{{\on{monot}}}
\newcommand\ns{{\on{ns}}}
\newcommand\st{{\on{st}}}

\newcommand\Mod[1]{\lVert #1 \rVert}

\begin{document}
\title{Removal of singularities in Morse-Bott symplectizations}
\author{Manav Gaddam and Sushmita Venugopalan}
\address{The Institute of Mathematical
  Sciences, Chennai, India and 
    Homi Bhabha National Institute, Mumbai, India}
\email{manavg@imsc.res.in, sushmita@imsc.res.in}

% \author{Sushmita Venugopalan}
% \address{The Institute of Mathematical
%   Sciences, Chennai, India}
% \address{Homi Bhabha National Institute, Mumbai, India}
% \email{sushmita@imsc.res.in}


\begin{abstract}
  Hofer-Wysocki-Zehnder \cite{hwz:deg} and Bourgeois
  \cite{bourg:thesis} proved that a finite energy punctured
  pseudoholomorphic curve in the symplectization of a Morse-Bott contact
  manifold either has a removable singularity or asymptotes to a Reeb
  orbit. We give an alternate proof of this result.
\end{abstract}

\maketitle
\section{Introduction}
On compact symplectic manifolds, singularities on finite energy
pseudoholomorphic maps are removable.  We consider symplectizations,
which are cylindrical symplectic manifolds, each of which is a product
of $\R$ and a Morse-Bott contact manifold.  A contact manifold is {\em
  Morse-Bott} if periodic Reeb orbits of the same period form a closed
submanifold, and the contact form is non-degenerate in the direction
normal to the submanifold. In such symplectizations, a singularity in
a finite energy pseudoholomorphic map may either be removable, or the
puncture may asymptote to a Reeb orbit.  In the special case where all
Reeb orbits of the contact form are non-degenerate (and consequently
isolated) the result was proved by Hofer-Wysocki-Zehnder
\cite{hwz:nondeg}. The proof for the general Morse-Bott case was
carried out by Hofer-Wysocki-Zehnder \cite{hwz:deg} and Bourgeois
\cite{bourg:thesis}, and is stated below. In this paper, we give an
alternate simpler proof of this result. 

\begin{theorem}\label{thm:main}
  Let $(M,\alpha \in \Om^1(M))$ be a compact contact manifold
  whose Reeb orbits form Morse-Bott families, and let $J$ be a
  cylindrical almost complex structure on $W:=\R \times M$ satisfying a
  taming condition from Definition \ref{def:acs} \eqref{part:acs-tame}. Let
  \[u:\R_{\geq 0} \times \R/\Z \to (W, J)\]
  be a $J$-holomorphic map with finite Hofer energy. Then, $u$ either 
  has a removable singularity at $\infty$ or there exist constants
  $c$, $\delta>0$, $a \in \R$, a period $T \in \R$ and a periodic Reeb
  orbit $\gamma : \R/T\Z \to M$ such that
  \begin{equation}
    \label{eq:decay2cyl}
    d_\cyl(u(s,t),\ol \gamma(s,t)) \leq c e^{- \delta s},  
  \end{equation}
  where
  \[\ol \gamma : \R \times \R/\Z \to W, \quad (s,t) \mapsto (a+ Ts,
    \gamma(Tt)) \]
  is a Reeb cylinder that lifts $\gamma$, and $d_\cyl$ is a product
  metric on $W=\R \times M$. Furthermore, if $J$ satisfies a
  compatibility condition from Definition \ref{def:acs}
  \eqref{part:acs-cptbl} and the period $T$ is known, there is a
  constant $\lam_T>0$, 
  which is the spectral gap at $0$ of an asymptotic operator 
  (defined in \eqref{eq:gamt-cyl}), such that for any
  $0<\delta<\lam_T$ and $u$ as above, there exists a
  constant $c$ for which the estimate \eqref{eq:decay2cyl} holds.
\end{theorem}

\noindent All the terms occurring in Theorem \ref{thm:main} are defined in Section \ref{sec:bg}. The above result does not hold if the Morse-Bott assumption is dropped, as shown by Siefring \cite{siefring:deg}.

We prove  Theorem \ref{thm:main} 
by transforming the map $u$ to a map with a removable singularity. To explain the main idea, 
consider first the case when $M$ is fibered by Reeb orbits of length $1$. Then, 
$W$ is a $\C^\times$ bundle. Suppose the map $u$ in Theorem \ref{thm:main} converges to an $N$-cover of a Reeb orbit as $s \to \infty$. 
Then the {\em twisted} map defined as $u_\tw(s,t):=e^{-N(s+it)}u(s,t)$ has a removable singularity at the infinite end. Exponential convergence of $u$ to the Reeb cylinder follows from exponential convergence of $u_\tw$ to $u_\tw(\infty)$. Thus twisting the map $u$ by the right amount transforms the problem into one of removable singularity. The proof of Theorem \ref{thm:main} in this special case is carried out in Section \ref{subsec:sympcase}.

Now consider the general case, where Reeb orbits occur in Morse-Bott families, and suppose the map $u=(u_\R,u_M)$ converges to a Reeb orbit of period $T$ as $s \to \infty$. Then, the
analogue of the twisting step produces a map
\[u_\tw(s,t):=(u_\R - sT, \psi_{-tT}  \circ u_M): \R_{\geq 0} \times [0,1] \to W,\]
where for any $\tau \in \R$, $\psi_\tau : M \to M$ is the Reeb flow by time $\tau$. 
The pseudoholomorphic map $u_\tw$ on a strip 
converges to a constant as $s \to \infty$, but it 
does not close up in the sense that $u_\tw(s,0) \neq u_\tw(s,1)$. To get a map with good boundary behaviour, we {\em double} the strip and define
\[v(s,t):=(u_\tw(s,t), u_\tw(s,1-t)): \R_{\geq 0} \times [0,\thh] \to W^2.\]
The doubled strip boundary maps to a pair of Lagrangians, namely the
diagonal $\Delta$ in $W^2$, and a shifted diagonal $\Delta_T$
consisting of points $((a,m), (a,\psi_T(m))$. The Lagrangians intersect cleanly because
the family of Reeb orbits is Morse-Bott.  The doubled strip $v$
converges to a point in the Lagrangian intersection. The problem is
thus transformed, via twisting and doubling, to one on removal of
singularity for strips in a compact manifold. The proof of Theorem
\ref{thm:main} in the general case is carried out in Section \ref{subsec:contcase}.

Sections \ref{sec:stripmonot} and \ref{sec:expconv} contain results
used in the proof of the main theorem. These are adaptations of
standard results, but we include the proofs in an attempt to make the
presentation as complete as possible.

\section{Background}\label{sec:bg}

We define the various terms that occur in the statement of Theorem
\ref{thm:main}.  The proof of the theorem is carried out in the
setting of {\em stable Hamiltonian structures} which generalize
contact manifolds, and are a natural setting for proving compactness
results in symplectic field theory. See \cite{bo:com} and \cite[Section 6.1]{wendl:sft}.

\begin{definition}
  \begin{enumerate}
  \item {\rm(Stable Hamiltonian structure)} A {\em stable Hamiltonian
      structure} is a triple $(M,\alpha,\Om)$ consisting of an
    odd-dimensional manifold $M$, a non-vanishing one-form
    $\alpha \in \Om^1(M)$ and a closed two-form $\Om \in \Om^2(M)$ that
    is non-degenerate on the tangent sub-bundle
    $\ker(\alpha) \subset TM$ and satisfies $\ker \Om \subseteq \ker (d\alpha)$.
  \item {\rm(Reeb vector field)} Given  a stable
    Hamiltonian structure  $(M,\alpha,\om)$, the {\em Reeb vector field} $R \in \Vect(M)$ on
    $M$ is defined by the conditions 
    \[R_\alpha \in \ker \Om, \quad \alpha(R_\alpha)=1. \]
    For any $t \in \R$ we denote by $\psi_t:M \to M$ the time $t$ flow
    by the Reeb vector field.
  \end{enumerate}
\end{definition}

\begin{example}
  \begin{enumerate}
  \item {\rm(Contact case)} A contact manifold $(M,\alpha)$ is a stable Hamiltonian structure with $\Om:=d\alpha$.
  \item {\rm(Symplectic case)} A special case of a stable Hamiltonian
    structure is when the Reeb vector field generates a free
    $S^1$-action on $M$, and consequently $\Om$ is a pullback of a 
    symplectic form on the quotient $M/S^1$.
  \end{enumerate}
  Motivated by these examples, in a stable Hamiltonian structure
  $(M,\alpha,\Om)$, we refer to the one-form $\alpha$ as the {\em contact
    form} and $\Om$ as the {\em horizontal symplectic form}.
\end{example}

\begin{definition}
  {\rm(Symplectization)}
  The {\em symplectization} of a
  stable Hamiltonian structure $(M,\alpha,\Om)$ is the product
  \[W:=\R \times M.\]
  %
  % equipped with a two-form
  % %
  % \[\om + d(\pi_\R \alpha)\]
  % %
  % which is a symplectic form in a neighborhood of $\{0\} \times M$. 
    {\em Translation} by $\tau \in \R$ on $W$ is denoted by
    \[e^\tau : W \to W, \quad (a,m) \mapsto (a+\tau,m).\]

\end{definition}
\noindent The term `symplectization' comes from the fact that in the contact case, $W=\R \times M$ has a symplectic form given by  $d(e^t\alpha)$, but there is no globally defined symplectic form for a general stable Hamiltonian structure. We will not use this form for defining the area of pseudoholomorphic curves.

Asymptotic decay results require us to restrict our attention to contact
forms where Reeb orbits are either isolated (non-degenerate case), or
occur in a family with a non-degeneracy condition in the normal
direction (Morse-Bott case).
\begin{definition}
  \begin{enumerate}
  \item {\rm(Non-degeneracy)} For any $T>0$, a $T$-periodic orbit
    $\gamma : \R/T\Z \to M$ of the Reeb vector field is {\em
      non-degenerate} if the derivative of the return map does not
    have $1$ as an eigenvalue in the transversal of the Reeb
    orbit, that is,
    \[d\psi_T(\gamma(0)) - \Id : T_{\gamma(0)}M/T\gamma \to
      T_{\gamma(0)}M/T\gamma\]
    is non-singular. The contact form $\alpha$
    is non-degenerate if all its Reeb orbits are non-degenerate.
  \item {\rm(Morse-Bott)} A contact form $\alpha$ is {\em Morse-Bott}
    if for any period $T>0$,
    \[N_T:=\{x \in M : \psi_T(x)=x, \enspace \psi_t(x) \neq x \enspace \forall t \in (0,T)\}\]
    is a submanifold of $M$ and $T_x N_T=\ker(d\psi_T(x) - \Id)$. 
  \end{enumerate}
\end{definition}

\begin{definition}\label{def:acs}
  {\rm(Almost complex structures)} Let $(M,\alpha,\Om)$ be a stable
  Hamiltonian structure. Let $W:=\R \times M$.
  \begin{enumerate}
  \item {\rm(Cylindrical)} An almost complex structure $J$ on $W$ is
    cylindrical if
    \begin{itemize}
    \item it is invariant under $\R$-translations,
    \item $J(\partial_r)=R_\alpha$ where $r$ is the $\R$-coordinate,
      and $R_\alpha$ is the Reeb vector field,
    \item and the sub-bundle $\ker(\alpha) \subset TM \subset TW$ is
      $J$-invariant.
    \end{itemize}
    The conditions imply that $J$ descends to an almost complex
    structure on the distribution $\ker(\alpha) \subset TM$, which we
    denote by $J|\ker(\alpha)$.
  \item {\rm(Horizontal tamedness)} \label{part:acs-tame}
    A cylindrical almost complex
    structure $J$ is {\em horizontally tamed} if the restriction
    $J|\ker(\alpha)$ is tamed by the two-form $\Om + rd\alpha$ for all
    $r \in [-\seps,\seps]$.
  \item {\rm(Horizontal compatibility)} \label{part:acs-cptbl}
    A cylindrical almost complex
    structure $J$ is {\em horizontally compatible} if the restriction
    $J|\ker(\alpha)$ is compatible with the two-form $\Om + rd\alpha$
    for all $r \in [-\seps,\seps]$.
  \end{enumerate}
\end{definition}

Finiteness of Hofer energy is a necessary condition for punctured
pseudoholomorphic maps in a symplectization to have good asymptotic
behavior.  To define the notion of Hofer energy of a curve in a
symplectization, we consider a family of embeddings of the
symplectization into a compact symplectic manifold, and take the
supremum of the symplectic area of the curve over all such
embeddings. This quantity is called ``energy'' in
\cite{Hofer:weinstein}, \cite{bo:com}, \cite[Section 9.2]{wendl:sft}.
\begin{definition}
  {\rm(Hofer Energy)}  Let $(M,\alpha,\Om)$ be a stable Hamiltonian structure. 
   \begin{enumerate}
  \item {\rm(Embeddings into a compact symplectic manifold)}
  Let $\seps>0$ be a small constant so that
 \begin{equation}
      \label{eq:wsymp}
      (W_\cpt,\om_\cpt):=([-\seps,\seps] \times M, \Om + d(\pi_\R\alpha)), 
    \end{equation}
     is a symplectic manifold. The constant $\seps$ is assumed to be a fixed
    constant for a given $M$.  
Let 
  \[ \cT:=\{\varphi:\R \to (-\seps,\seps) : \varphi \text{ is an increasing diffeomorphism}\}.\]    
  For any $\varphi \in \cT$, define an embedding
  \[(\varphi \times \Id_M) : W \to W_\cpt.\]

\item {\rm(Hofer energy)} 
    Let $C$ be a Riemann surface. 
  The {\em Hofer energy} of a map
  $u:C \to \R \times M$ is
  \[E_\Hof(u):= \sup_{\varphi \in \cT}\int_C u^*\om_\varphi,\]
  where 
  %
%  \begin{equation}
 %   \label{eq:omphi}
  \[  \om_\varphi:=(\varphi \times \Id_M)^*\om_\symp=\Om + d(\varphi \alpha). \]
%  \end{equation}
  %
  The {\em horizontal energy} of $u$ is defined as
  \[E_\hor(u):=\int_C u^*\Om.\] 
\end{enumerate}
\end{definition}

A horizontally tamed (resp. compatible) cylindrical almost complex
structure $J$ on $\R \times M$ is $\om_\varphi$-tamed
(resp. $\om_\varphi$-compatible) for any increasing diffeomorphism
$\varphi \in \cT$. Therefore, for a non-constant $J$-holomorphic curve
$u$, the $\om_\varphi$-area is positive, and $E_\Hof(u) \geq 0$.  For
such a $J$-holomorphic curve $u$, the horizontal energy is bounded by
Hofer energy
\begin{equation}
  \label{eq:horhof}
  E_\hor(u) \leq E_\Hof(u).
\end{equation}
This is a consequence of \cite[Lemma 9.11]{wendl:sft}.

\begin{assumption}\label{ass:metric}
  {\rm(Metric)} We assume that both $W=\R \times M$ and $W_\symp=[-\seps,\seps] \times M$ are  equipped with a product metric $g_\cyl=dt^2 \oplus g_M$, and $d_\cyl$ is the distance function with respect to $g_\cyl$.  
\end{assumption}

\section{Proof of the theorem}
We prove Theorem \ref{thm:main} in this section. 
Throughout the section, $M$ has a stable Hamiltonian structure, $J$ is a
horizontally tamed cylindrical almost complex structure on
$W:=\R \times M$ and
\[u=(u_\R,u_M) : \R_{\geq 0} \times S^1 \to \R \times M = W\]
is a $J$-holomorphic map. We assume throughout that the coordinate on $S^1$ is given by the identification $S^1 \simeq \R/\Z$.
\subsection{A preliminary limit}
For a punctured pseudoholomorphic map with finite Hofer energy, the punctured end approaches a 
`preliminary limit'
in the following sense.
\begin{proposition}\label{prop:sk}
  Suppose $E_\Hof(u)<\infty$. Then,
  \begin{enumerate}
  \item {\rm(\cite[Lemma 9.3]{wendl:sft})} $|du|_{L^\infty}<\infty$,
  \item {\rm(\cite[p165]{wendl:sft})}
    \label{part:sk2}
    and there are sequences
    $\tau_k \in \R$, $s_k \to \infty$ such that the sequence of maps
    \[(s,t) \mapsto e^{\tau_k}u(s+s_k,t) \]
    converges uniformly on compact subsets to a Reeb cylinder
    \[\ol \gamma_\pre: \R \times S^1 \to W, \quad (s,t) \mapsto (Ts, \gamma_\pre(Tt)), \]
    where $\gamma_\pre: \R/T\Z \to M$ is a $T$-periodic Reeb orbit for
    some $T \in \R$. In case $T=0$, $\gamma_\pre$ is a constant point
    in $M$.
  \end{enumerate}
\end{proposition}
\noindent We call $\gamma_{pre}$ the {\em preliminary limit} of the cylinder $u$.
 % We
% next prove that if the preliminary limit $\gamma_{pre}$ is constant,
% then the map $u$ has a removable singularity.

\subsection{The case of a removable singularity}
We show that in the special case where the preliminary limit obtained in Proposition \ref{prop:sk} is a constant, the map $u$ has a removable singularity at the puncture $s=\infty$.

\begin{proposition}\label{prop:skref}
  {\rm(Removable singularity)}
  Suppose the map $u:\R_{\geq 0} \times S^1 \to W$ has finite horizontal energy, that is, $\int_{[0,\infty) \times S^1}u^*\Om <\infty$, 
  and there is a sequence $s_k \to \infty$ for which the maps
  \[u_M(s_k,\cdot) : S^1 \to M\]
   $C^1$-converge to a constant map.  Then,
  there is a point $w \in \R \times M$ and constants $s_0>0$, $c_u>0$ such that 
  \begin{equation}
    \label{eq:remdecay}
    d_\cyl(u(s,t),w) \leq c_ue^{-s} \quad \forall s \geq s_0, t \in S^1.   
  \end{equation}
\end{proposition}
We first prove an intermediate result on Hofer energy decaying to zero on the end of the domain cylinder.
\begin{lemma}\label{lem:hof0}
  Let $u$ be a map satisfying the hypothesis of Proposition \ref{prop:skref}. Then,
  \begin{equation}
    \label{eq:hoflim}
    \lim_{s \to \infty}E_\Hof(u, [s,\infty) \times S^1)=0.
  \end{equation}
\end{lemma}
\begin{proof}
  Recall that the Hofer energy $E_\Hof$ of a map $u$ is the supremum of the $\om_\varphi$-area of $u$, where $\varphi : \R \to (-\seps,\seps)$ ranges over increasing diffeomorphisms and $\om_\varphi=\Om + d(\varphi \alpha)$.
  To prove the lemma, we focus on the term $d(\varphi \alpha)$ since $\int_{[0,\infty) \times S^1} u^*\Om$ is finite, and therefore, $\int_{[s,\infty) \times S^1} u^*\Om \to 0$ as $s \to \infty$. 
    For any $\eps$, there exists $k$ such that for
  any $l \geq k$ and any increasing diffeomorphism $\varphi$,
  \[\int_{[s_k,s_l] \times S^1}u^*d(\varphi \alpha)=\int_{\partial[s_k,s_l]
      \times S^1}(\varphi \circ u_\R)(u_M^*\alpha) <\eps.\]
  Indeed $\varphi \circ u_\R$ is uniformly bounded, and since the sequence $u_M(s_k,\cdot)$ converges to a constant, 
  $|u_M^*\alpha| \to 0$ uniformly on $\{s_k\} \times S^1$ as
  $k \to \infty$. Consequently $\int_{[s_k,\infty) \times S^1} u^*d(\varphi \alpha) \leq \eps$, and \eqref{eq:hoflim} follows.
\end{proof}

\begin{remark}
  \label{rem:hofstr}
  For future reference, we note that Lemma \ref{lem:hof0} also 
holds for strips. The statement is as follows: 
Suppose $u:\R_{\geq 0} \times [0,1] \to W$ satisfies the hypotheses of Proposition \ref{prop:skref}, namely, the horizontal energy of $u$ is finite, and a sequence of maps $u_M(s_k,\cdot): [0,1] \to M$ converges to a constant. Then, \[\lim_{s \to \infty}E_\Hof(u, [s,\infty) \times S^1)=0.\]
The proof is identical to that of Lemma \ref{lem:hof0}.
\end{remark}

The main ingredient in the proof of Proposition \ref{prop:skref} is Gromov's monotonicity result for pseudoholomorphic curves (see for example \cite[Proposition 4.3.1]{Sikorav}), which we now state. A set $U_\J$ of $\om$-tamed almost complex structures on $X$ is said to be a {\em $C^0$-neighborhood} if it is bounded in the $C^0$-topology, that is, there exists $\eps>0$ such that
\[\Mod{J_1 - J_2}_{C^0(X)}<\eps \quad \forall \, J_1, J_2 \in U_\J.\]
\begin{proposition} \label{prop:monohol} {\rm(Monotonicity)} Let
  $(X,\om)$ be a compact symplectic manifold with a Riemannian metric $g$, and let $U_\J$ be a
  $C^0$-neighborhood on the space of $\om$-tamed almost complex structures.
  There exist constants $c_{\monot}, r_\monot>0$ such that for any $x \in X$,
  $0 < r \leq r_\monot$, a Riemann surface $C$ with boundary $\partial C$
  and a pseudoholomorphic map $u:C \to X$ with respect to a
  domain-dependent almost complex structure $J:C \to U_\J$ whose image
  contains $x$ and $u(\partial C) \subset \partial B_g(x,r)$,
  \[\int_C u^*\om \geq c_{\monot}r^2. \]
\end{proposition}

\begin{proof}[Proof of Proposition \ref{prop:skref}]
  We first show that the image of $u$ is bounded; the proof is via the monotonicity result which says that 
  the image can not ``wander too much'', since the Hofer energy decays to zero near the puncture. 
  Let $c_\monot, r_\monot>0$ be the constants
   obtained by applying the monotonicity result to  
  the compact symplectic manifold 
  \[(W_\symp, \om_\symp )=([-\seps,\seps] \times M, \Om +d(\pi_\R \alpha),J).   \]
  Choose a small $\eps<\seps, r_\monot$.  Since the maps
  $u_M(s_k,\cdot)$ converge to a constant in $M$, Lemma \ref{lem:hof0}
  says that $E_\Hof(u,[s,\infty) \times S^1)$ goes to zero as $s \to \infty$.
  Let $\sig_0$ be such that
  \begin{equation}
    \label{eq:hofbd}
    E_\Hof(u,[\sig_0,\infty) \times S^1)<c_{\monot}\eps^2.  
  \end{equation}
  Next consider a domain point $z_0 \in [\sig_0,\infty) \times S^1$. Let
  \[a:=u_\R(z_0), \quad m:=u_M(z_0).\]
  Choose an increasing diffeomorphism $\varphi_a : \R \to (-\seps,\seps)$ that satisfies
  \begin{equation}
    \label{eq:phia}
    \varphi_a(s)=s-a \quad \text{on $(a-\eps, a+\eps)$}.   
  \end{equation}
  Then $\Phi_a:=\varphi_a \times \Id_M: W \to W_\symp$ is an isometry on $B_\eps(a,m)$, where
  both $W$ and $W_\symp$ are equipped with the cylindrical metric (Assumption \ref{ass:metric}). 
   Let $C_{z_0} \subset (\Phi_a \circ u)^{-1}(\ol B_\eps(0,m))$
   be the connected component containing $z_0$.
   If $C_{z_0}$ is compact, then $C_{z_0}$ intersects $\{\sig_0\} \times S^1$. Indeed, otherwise, the monotonicity result (Proposition \ref{prop:monohol}) applied to the map
  $\Phi_a \circ u$ and the ball $B_\eps(0,m)$ implies 
%   The map $(\Phi_a \circ u)|_{C_{z_0}}$ is $J$-holomorphic and has an area bound 
  %
  \[c_\monot \eps^2 < \int_{C_{z_0}} (\Phi_a \circ u)^*\om_\symp = \int_{C_{z_0}} u^*\om_{\varphi_a} \leq E_\Hof(u,C_{z_0}),\]
  which contradicts \eqref{eq:hofbd}. Therefore, one of the following two cases occur:
  \begin{description}
  \item[Case 1] For some $z_0 \in [\sig_0,\infty) \times S^1$, 
    $C_{z_0}$ is non-compact. In this case, $C_{z_0}$ intersects
    $\{s\} \times S^1$ for all $s \geq s(z_0)$.  Since $|du| \leq c_0$ for some constant $c_0$,
    we conclude that the image $u([s(z_0),\infty) \times S^1)$ lies
    within a $c_0$-radius of the neighborhood $B_\eps(a,m)$, and
    therefore the image of $u$ is bounded.
  \item[Case 2] The set $C_{z_0}$ is compact for all $z_0 \in [\sig_0,\infty) \times S^1$.
     In this case, for
  any such $z_0$, $C_{z_0}$ intersects $\{\sig_0\} \times
  S^1$.  Consequently,
  \[d_{\cyl}(u(z_0),u(\{\sig_0\} \times S^1))<\eps \quad \forall z_0 \in [\sig_0,\infty) \times S^1,\]
   and therefore, the image of $u$
  is bounded in $\R \times M$.  
  \end{description}

  We now finish the proof of the proposition. Since the image of $u$
  is bounded, there is an interval $[a_0, a_1] \subset \R$ such that
  the image of $u$ is contained in $[a_0,a_1] \times M$.  For any
  increasing diffeomorphism $\varphi : \R \to (-\seps,\seps)$,
  $\om_\varphi$ is a taming symplectic form on the compact manifold
  $[a_0,a_1] \times M$.  The area
  $\int_{\R_{\geq 0} \times S^1}u^*\om_\varphi$ is bounded by
  $E_\Hof(u)$, and is therefore finite.  We apply the removal of
  singularities theorem for compact manifolds \cite[Theorem 4.1.2]{ms:jh} on a reparametrization
  of $u$, namely
  \[v: B_1 \bs \{0\} \to [a_0,a_1] \times M, \quad v(re^{i\theta}):=u(\tfrac {-\ln r}{2\pi}, \tfrac {-\theta}{2\pi}), \]
  where $B_1 \subset \C$ is the unit ball, and conclude that $v$
  extends holomorphically over $0$.  The differentiability of $v$
  implies that $d_\cyl(v(z),v(0)) \leq c|z|$ for some constant
  $c$. Changing to cylindrical coordinates, we obtain the exponential
  decay \eqref{eq:remdecay} stated in the proposition.
\end{proof}

\subsection{The symplectic case}\label{subsec:sympcase}
In this section, we state and prove the analog of Theorem
\ref{thm:main} in the symplectic case, which will motivate our proof
in the general case. We recall that a stable Hamiltonian structure $M$
is symplectic if the Reeb vector field generates a free $S^1$-action
on $M$.  The $S^1$-action extends to a $\C^\times$-action on $W$ given
by 
\begin{equation}
  \label{eq:cw-act}
\C^\times \times W \to W, \quad   (e^{\sig+i\tau},(a,m)) \mapsto (a+\sig , e^{i\tau}m).     
\end{equation}
If the cylindrical almost complex structure $J$ on $W$ is $\C^\times$-invariant, then the $\C^\times$-action is $J$-holomorphic. Note that this is a stronger condition than the cylindricity 
of $J$, because in general, a cylindrical almost complex structure is
not invariant under Reeb flow. We also note that in this case a cover of a $\C^\times$-orbit
\[\R \times S^1 \ni (s,t) \mapsto e^{T(s+it)}w, \quad w \in W, \enspace T \in 2 \pi \Z\]
is a Reeb cylinder. 

\begin{proposition}\label{prop:symp-analog}
  {\rm(Symplectic analog of Theorem \ref{thm:main})} Suppose the
  stable Hamiltonian structure $(M,\alpha, \Om)$ is symplectic, and
  suppose the almost complex structure $J$ on $W$ is
  $\C^\times$-invariant. Let
  \[u=(u_\R, u_M):\R_{\geq 0} \times S^1 \to (W, J)\]
  be   a $J$-holomorphic map
  with finite Hofer energy. Then, there exist constants $T \in 2\pi\Z$, $c_u>0$, and a point $w \in W$ such that
   \begin{equation}
     \label{eq:decay-symp}
       d_\cyl(u(s,t), e^{T(s+it)}w) \leq c_u e^{- s}.
  \end{equation}
\end{proposition}
\begin{proof}
  We prove the result by transforming the given map $u$ to a map with a removable singularity.
  First, we observe that
  by Proposition \ref{prop:sk}, there is a sequence $s_k \to \infty$
  such that the sequence
  \[u_M(s_k,\cdot) : \R/\Z \to M\]
  converges to a cover of a Reeb orbit, which in this case, is an
  orbit of the $S^1$-action. That is, the sequence $u_M(s_k,\cdot)$
  converges to $\R/\Z \ni t \mapsto e^{iT t}m $ for some
  $T \in 2 \pi \Z$ and $m \in M$.  The twisted cylinder
  \begin{equation}
    \label{eq:symp-twist}
    u_\tw:\R_{\geq 0} \times \R/\Z \to \R \times M, \quad (s,t) \mapsto
    e^{-T(s+it)}u(s,t)  
  \end{equation}
  is $J$-holomorphic,
  because the $\C^\times$-action on $W$ is $J$-holomorphic (see \eqref{eq:cw-act}). We claim $u_\tw$
  has a removable singularity at the infinite end of the cylinder.
  Indeed, Proposition \ref{prop:skref}
  is applicable on $u_\tw$, as the sequence
  of maps $\pi_M \circ u_\tw(s_k,\cdot)$ converges to a constant
  $m \in M$, and $u_\tw$ has finite horizontal energy since
  \[E_\hor(u_\tw)=E_\hor(u) \leq E_\Hof(u).\]
  The removable singularity result (Proposition \ref{prop:skref}) implies
  that $u_\tw(\infty)=w \in W$ for some $w \in W$ and there exists a constant $c_u$
  such that for all domain points $(s,t)$, 
  \begin{equation}
    \label{eq:utw-decay}
    d_\cyl(u_\tw(s,t),w) \leq  c_ue^{-s},  
  \end{equation}
  see \eqref{eq:remdecay}. Since $u_\tw$ is the $e^{-T(s+it)}$-twist
  of $u$, the estimate \eqref{eq:decay-symp} follows from \eqref{eq:utw-decay},
  finishing the proof of Proposition \ref{prop:symp-analog}.
 \end{proof}

\subsection{The contact case}\label{subsec:contcase}
In this section, we prove Theorem \ref{thm:main} in the general case. Although the result is stated for contact manifolds, our proof applies to any stable Hamiltonian structure.

% \begin{remark}
%   The proof of the contact case differs from the symplectic case because there is no $\C^\times$-action on $W=\R \times M$.
% The additive group $\R^2$, which is the universal cover of $\C^\times$, has an action on $W$  given by
%   \begin{equation}
%     \label{eq:r2w}
%     \R^2 \times (\R \times M) \to \R \times M, \quad (\sig,\tau) \cdot (a,m) \mapsto (a+\sig, \psi_\tau(m)),
%   \end{equation}
%   %
%   where $\psi_\tau$ is the flow of the Reeb vector field on $M$ by time $\tau$.
%   However,  for a cylindrical almost complex structure $J$ on $W$, 
%   the action \eqref{eq:r2w} is not $J$-holomorphic because $\psi_\tau$ does not preserve $J$.%  To obtain a $J$-holomorphic action, we consider a family of almost complex structures on $W$: In particular, the action of $\R^2$ on $\R^2 \times W$ given by
%   % %
%   % \[\R^2 \times (\R^2 \times W) \to \R^2 \times W, \quad (\sig, \tau) \cdot ((s,t), (a,m)) \mapsto ((s+\sig, t+\tau), (a+\sig, \psi_\tau(m)))\]
%   % %
%   % is $J$-holomorphic, where the complex structure on $\R^2 \times W$ is defined as
%   % \[J_{((s,t),(a,m))}^{\R^2 \times W}:=J^{\R^2}_{(s,t)} \oplus (\psi_t^*J^W)_{(a,m)}.\]
% \end{remark}

\begin{proof}
  [Proof of Theorem \ref{thm:main}] We describe the idea of the proof.
  By Proposition \ref{prop:sk}, the map $u$ has a preliminary limit,
  which means that there is a sequence $s_k \to \infty$ and a sequence
  $\tau_k \in \R$ such that the sequence of maps
  \begin{equation}
    \label{eq:tauk-conv}
    \R \times \R/\Z \ni (s,t) \mapsto e^{\tau_k}u(s+s_k,t) \in W  
  \end{equation}
  converges to a Reeb cylinder
  \[(s,t) \mapsto (Ts, \gamma_\pre(Tt)) \in \R \times M, \]
  where 
  $\gamma_\pre: \R/\Z \to M$ is a Reeb orbit of period $T \in \R$.  Analogous to
  the proof of the symplectic case, the first step of our proof is to
  twist the map $u$ by the inverse of the time $T$ Reeb flow.  If we
  copy the twisting step from \eqref{eq:symp-twist} we end up with a
  pseudoholomorphic strip
  \begin{equation}
    \label{eq:utw-strip}
    u_\tw: \R \times [0,1] \to \R \times M, \quad (s,t) \mapsto
    (u_\R(s,t)-Ts,\psi_{-Tt}(u_M(s,t)))
  \end{equation}
  that does not close up to form a cylinder, since
  $u_\tw(s,0) \neq u_\tw(s,1)$.
  The map $u_\tw$ is holomorphic with respect to the domain-dependent almost complex structure
  $J(s,t):=\psi_{Tt}^*J$, because by the $\R$-invariance of $J$, \label{page:jhol}
  \[\partial_s u_\tw= d\psi_{Tt}(\partial_s u) - T \partial_r, \quad \partial_t u_\tw= d\psi_{Tt}(\partial_t u) - T R_\alpha,  \]
  and $J(\partial_r)=R_\alpha$ where $r$ is the $\R$-coordinate on $\R \times M$. 
  The convergence in \eqref{eq:tauk-conv} implies that the
  maps
  \begin{equation}
    \label{eq:utw-conv}
    (s,t) \mapsto e^{\tau_k}u_\tw(s+s_k,t) \in W   
  \end{equation}
  converge uniformly on compact subsets to a constant $(0,m) \in W$
  where $m:=\gamma_\pre(0)$. However, $u_\tw$ does not satisfy
  reasonable boundary conditions that would let us apply results of
  pseudoholomorphic maps.

  We `double' the strip $u_\tw$ to produce a pseudoholomorphic strip in the
  product $(\R \times M)^2$ whose boundary lies on a pair of cleanly
  intersecting Lagrangian submanifolds.  The definition of the double
  is standard, see for example \cite{hutchings:mo}, and in our case it
  is given by
  \begin{equation}
    \label{eq:vdef}
  v:\R \times [0,\thh] \to (\R \times M)^2 \quad (s,t) \mapsto
  (u_\tw(s,t),u_\tw(s,1-t)).   
  \end{equation}
  The map $v$ is pseudoholomorphic with respect to the $t$-dependent almost
  complex structure
  \begin{equation}
    \label{eq:Jtdef}
    \J := (J_t)_{t \in [0,\hh]}, \quad J_t:=\psi_{Tt}^*J \oplus
  (-\psi_{T(1-t)}^*J).  
  \end{equation}
  For any increasing diffeomorphism $\varphi : \R \to [-\seps,\seps]$,
  the almost complex structure $J_t$ is tamed by the symplectic form
  $\om_\varphi \oplus -\om_\varphi$.  Furthermore, the boundary
  $v(\cdot,\hh)$ maps to the diagonal
  \[\Delta:=\{(x,x): x \in \R \times M\},\]
  and $v(\cdot,0)$ maps to the shifted diagonal
  \[\Delta_T:=\{(\psi_T(x),x): x \in \R \times M\},\]
  since by the definition \eqref{eq:utw-strip} of the twisted map, 
  $u_\tw(s,1)=\psi_{-T}(u_\tw(s,0))$.
  \label{page:diag}
  The submanifolds $\Delta$ and
  $\Delta_T$ are Lagrangian submanifolds in $(\R \times M)^2$ with
  respect to the symplectic form $\om_\varphi \oplus -\om_\varphi$ for
  any $\varphi$.  The Morse-Bott condition on the contact form on $M$ implies that the
  Lagrangians $\Delta$, $\Delta_T$ intersect {\em cleanly} in
  $(\R \times M)^2$, that is, $\Delta \cap \Delta_T$ is a submanifold
  and
  \[T_x(\Delta \cap \Delta_T)=T_x\Delta \cap T_x\Delta_T \quad \forall
    x \in \Delta \cap \Delta_T.\]
  In the rest of the proof, we show that the image of $v$ is bounded
  using a strip version of the monotonicity theorem that is stated and
  proved in Section \ref{sec:stripmonot}.  Once the image is shown to
  be bounded, we may apply results for the removal of singularity for
  strips in compact symplectic manifolds. This result is standard, but
  we outline the proof for the sake of completeness. In particular, we
  will show that $v$ converges to a point $p \in \Delta \cap \Delta_T$
  at a rate that is exponential in the $\R$-coordinate. The details
  are as follows.

  \vskip .1in {\sc Step 1} : {\em The image of the doubled strip $v$ is bounded.} \\
  We closely follow the proof of Proposition \ref{prop:skref}, which
  is the analogous result for cylinders.
The twisted strip $u_\tw$
  satisfies
  \begin{equation}
    \label{eq:twhof}
    \lim_{s \to \infty}E_\Hof(u_\tw,[s,\infty) \times [0,1])=0
  \end{equation}
  as a consequence of Remark \ref{rem:hofstr}. The hypotheses required by Remark \ref{rem:hofstr} are indeed satisfied, as the horizontal energy of $u_\tw$ is finite since by \eqref{eq:horhof},
  \[\int_{\R_{\geq 0} \times [0,1]} u_\tw^* \Om = \int_{\R_{\geq 0} \times S^1}u^*\Om \leq E_\Hof(u) <\infty; \]
  and since the sequence of maps $(\pi_M \circ u_\tw)(s_k,\cdot) : [0,1] \to M$ converges in $C^1$ to the constant map $m \in M$ (see \eqref{eq:utw-conv}).

  The proof of Step 1 is by applying the strip monotonicity result to the doubled strip $v$. 
  We first describe the spaces on which  the strip monotonicity result (Theorem \ref{thm:monot}) is applied.  The convergence \eqref{eq:utw-conv} of the translations of $u_\tw$ implies that 
  \begin{equation}
    \label{eq:tauconv}
    e^{\tau_k}v(s_k,\cdot) \to ((0,m),(0,m)) \quad \text{in $C^\infty([0,\thh],W^2)$}.
  \end{equation}
  The strip monotonicity result (Theorem \ref{thm:monot}) is applied to
  the symplectic manifold
  \[(W^2, \tilde \om_\varphi), \quad \tilde \om_\varphi:= \om_\varphi \oplus
    (-\om_\varphi)\]
  where $\varphi:\R \to (-\seps,\seps)$ is an increasing
  diffeomorphism  fixed for the rest of the proof; the pair
  of Lagrangians $\Delta$, $\Delta_T$ in $W^2$; and the intersection
  point $((0,m),(0,m)) \in \Delta \cap \Delta_T$.  Let $R$, $c_\monot$
  be the constants obtained from the strip monotonicity result.

  We define a sequence of subdomains $C_k$ on the strip. 
  Choose
  $\eps<R$. By \eqref{eq:twhof}, there exists
  $\sig_0>0$ be such that
  \begin{equation}
    \label{eq:utwhof}
    E_\Hof(u_\tw,[\sig_0,\infty) \times [0,1])<c_\monot\eps^2.   
  \end{equation}
  Choose any $t_k \in (0,\hh)$, and let
  $z_k:=(s_k,t_k) \in \R_{\geq 0} \times [0,\hh]$. By
  \eqref{eq:tauconv}, we may assume that $e^{\tau_k}v(z_k)$ lies in an
  $\frac {\eps} 2$-neighborhood of $((0,m),(0,m))$.  Let
  \[C_k:=\text{ connected component of $v^{-1}(\ol B_{\eps/2}(v(z_k)))$
      containing $z_k$} \subset \R_{\geq 0} \times [0,\thh]. \]
  % 
  % $C_{z_k}$ be
  % the connected component of $v^{-1}(B_{\eps/2}(v(z_k)))$ that
  % contains $z_k$.
  The proof of the boundedness of $v$ splits up into the following two
  cases.

  \vskip .1in  {\sc Case 1} : {\em The domain $C_k$ is non-compact for some $k$.}
  \\
  Since $C_k$ is non-compact and closed in $\R_{\geq 0} \times [0,\hh]$, it intersects the segment
  $\{s\} \times [0,\hh]$ for all $s \geq s_k$. There is a uniform
  bound $|dv|_{L^\infty}<c_0$, since $|du|_{L^\infty}$ is finite and the
  Reeb flow $\psi_t$, $t \in [0,T]$ on $M$ alters the derivative of
  $u$ by a bounded multiplicative factor and  a bounded additive term. Consequently the image
  $v(\R_{\geq s_k} \times [0,\hh])$ is contained in a
  $\frac {c_0} 2$-neighborhood of $v(C_k)$.  Since
  $v(C_k) \subset \ol B_{\eps/2}(v(z_k))$, we conclude that the image
  of $v$ is bounded in $W^2$.

  \vskip .1in {\sc Case 2} : {\em The domain $C_k$ is compact for all $k$.}\\
  In this case we will show that for any $k$, $C_k$ intersects
  $\{\sig_0\} \times [0,\hh]$.  For the sake of contradiction, assume
  that $C_k$ lies in $\{s>\sig_0\}$. We proceed to apply the strip
  monotonicity theorem on the map $(e^{\tau_k}v)|C_k$ mapping to
  the ball $B_{\eps/2}(e^{\tau_k}v(z_k))$, which is contained in
  $B_{R}((0,m),(0,m))$.
  Since $C_k$ is disjoint from $\{0\} \times [0,\hh]$, the non-strip 
  boundary   $\partial C_k \bs (\R_{\geq 0} \times \{0,1\})$ is mapped
  by $e^{\tau_k}v$ 
  to $\partial
  B_{\eps/2}(e^{\tau_k}v(z_k))$. % On $W^2$, we use the symplectic form
  % $\tilde \om_\varphi:=\om_\varphi \oplus (-\om_\varphi)$ where the
  % diffeomorphism
  % $\phi : \R \to (-\seps,\seps)$ is the identity map on the interval
  % $[-\eps,\eps] \subset \R$.
  The strip monotonicity result (Theorem \ref{thm:monot}) then implies that
  \[\int_{C_k}(e^{\tau_k}v)^*\tilde \om_\varphi \geq c_\monot \eps^2,\]
  which contradicts \eqref{eq:utwhof} because
  \[\int_{C_k}(e^{\tau_k}v)^*\tilde \om_\varphi=\int_{\tilde 
      C_k}(e^{\tau_k}u_\tw)^*\om_\varphi \leq
    E_\Hof(u_\tw,[\sig_0,\infty) \times [0,1]), \]
  where
  $\tilde C_k:=\{(s,t) \in \R \times [0,1]: (s,t) \text{ or } (s,1-t) \in
  C_k\}$.  Therefore, our assumption is wrong and $C_k$ intersects
  $\{\sig_0\} \times [0,\hh]$ for all $k$.

  We finish the proof of Step 1 in case 2 by showing that the image of $v$ is contained within a
  bounded distance of the compact set $v(\{s=\sig_0\})$. Consider  any $z=(s,t) \in [\sig_0, \infty) \times [0,\hh]$. Choose $s_k>s$. Since $C_k$ is a connected set containing $z_k=(s_k,t_k)$ and intersecting $\{\sig_0\} \times [0,\hh]$, there exists a point $(s,t') \in C_k$. As we mentioned in Case 1, there is a uniform  bound $|dv|_{L^\infty}<c_0$, and so, $v(s,t) \in B_{c_0/2}(v(s,t'))$. Since $(s,t') \in C_k$, we have 
  $v(s,t') \in B_\eps(v(\{s=\sig_0\}))$. We conclude that $v(z)$ lies in a neighborhood of $v(\{s=\sig_0\})$ of radius
  $\eps + c_0/2$. \label{page:detail}
  % Hence, the set
  % $B_{\eps/2}(v(\{s=\sig_0\}))$ intersects $v(\{s=s'\})$ for all
  % $s' \geq \sig_0$. As we mentioned in Case 1, there is a uniform
  % bound $|dv|_{L^\infty}<c_0$, and so $v(\{s \geq \sig_0\})$ is
  % contained in a neighborhood of $v(\{s=\sig_0\})$ of radius
  % $\eps/2 + c_0/2$.

  \vskip .1in {\sc Step 2} : {\em There is a point $(a,m) \in W$ such
    that $\lim_{s \to \infty} v(s,t)=((a,m),(a,m))$. }
  The boundedness of the image of $v$ implies that the convergence of
  the loops $v(s_k,\cdot)$ in \eqref{eq:tauconv} can be refined. In
  particular, there is a constant $a \in \R$ such that after passing
  to a subsequence of $\{s_k\}_k$,
  \[ v(s_k,\cdot) \to w_0:=((a,m),(a,m)) \quad \text{in
    }C^\infty([0,\hh],W^2) \quad \text{as $k \to \infty$}. \]
  To prove Step 2, we apply the strip monotonicity result again to
  show that there is not enough energy at the end of the domain to repeatedly
  ``exit'' a neighborhood of $w_0$.  Suppose applying
  the strip monotonicity result (Theorem \ref{thm:monot}) on
  $(W^2,\om_\varphi \oplus (-\om_\varphi))$ at the Lagrangian
  intersection point $w_0 \in \Delta \cap \Delta_T$ produces constants
  $R$, $c_\monot$. Suppose the claim in Step 2 is not true. Then,
  there is a constant $0<r<R/2$ such that for any $k$ there is a point
  $z_k'=(s_k',t_k)$ on the strip $\R \times [0,\hh]$ with $s_k'>s_k$
  and such that $v(z_k') \in \partial B_r(w_0)$.  Choose $k_0$ such
  that
  \begin{equation}
    \label{eq:k0cond}
    v(s_k,\cdot) \subset B_{r/2}(w_0) \enspace \forall k \geq k_0, \quad E_\Hof(u_\tw,\{s>s_{k_0}\})<c_\monot r^2/4.   
  \end{equation}
  The latter  condition implies that
  \begin{equation}
    \label{eq:omWbd}
    \int_{[s_{k_0},\infty) \times [0,1/2]}v^*\tilde \om_\varphi< c_\monot r^2/4.  
  \end{equation}
  Choose $k_1>k_0$ such that $s_{k_1} > s'_{k_0}$. 
  We will produce a lower bound on $\tilde \om_\varphi$-area using the facts
  \[v(s_{k_0},\cdot) \subset B_{r/2}(w_0), \quad v(s_{k_1},\cdot) \subset B_{r/2}(w_0), \quad v(z_{k_0}') \in \partial B_r(w_0). \]
  % We will apply the strip monotonicity lemma on the ball
  % $B_{r/2}(v(z_{k_0}')) \subset W^2$, and on the domain
  % $C' \subset \R \times [0,1/2]$ defined as
Define 
  \[C':=\text{ connected component of
      $v^{-1}(\ol B_{r/2}(v(z_{k_0}')))$ containing $z_{k_0}'$}.\]
Since 
  $C' \subset (s_{k_0},s_{k_1}) \times [0,\hh]$, the non-strip boundary 
  $\partial C' \bs (\R \times \{0,\hh\})$ is mapped by $v$ to 
  $\partial B_{r/2}(v(z_{k_0}'))$. The strip monotonicity result then
  implies that
  \[\int_{C'}v^*(\om_\varphi \oplus (-\om_\varphi)) > c_\monot r^2/4,\]
  which contradicts \eqref{eq:omWbd}. Therefore, the claim in Step 2
  follows.

  \vskip .1in {\sc Step 3} : {\em Finishing the proof.}\\
  First, consider the case when the cylindrical almost complex
  structure on $W$ is horizontally compatible as in Definition
  \ref{def:acs} \eqref{part:acs-cptbl}.  We apply the exponential
  convergence result (Proposition \ref{prop:expdecay}) to the map
  $v:[0,\hh] \to W^2$ and any $0<\delta<\lam_T$, and obtain a constant
  $c_u$ such that
  \[d_\cyl(v(s,t),((a,m),(a,m))) \leq c_ue^{-\delta s}\]
  for all $s \geq 0$. Since $v$ is the double of $u_\tw$, we get
  \begin{equation}
    \label{eq:cyl-exp}
  d_\cyl(u_\tw(s,t),(a,m)) \leq c_ue^{-\delta s}.   
  \end{equation}
  The map $u$ can be recovered from $u_\tw$ as
  $u(s,t)=(u_{\tw,\R} + Ts, \psi_{tT}(u_{\tw,M}))$, where
  $u_\tw=(u_{\tw,\R}, u_{\tw,M})$.  The cylindrical metric from
  Assumption \ref{ass:metric} is $\R$-invariant, and on the compact
  manifold $M$, the Reeb flow alters the metric $d_M$ by a bounded
  amount. That is, there is a constant $c$ such that for a pair of
  points $p, q \in M$, and $0 \leq t \leq T$,
  \[d_M(\psi_t(p), \psi_t(q)) \leq c d_M(p,q). \]
  \label{page:reeb}
  Therefore, \eqref{eq:cyl-exp} implies that $u$ 
  converges exponentially to the Reeb cylinder
  $\ol \gamma(s,t):=(a+Ts, \psi_{tT}(m))$, that is, 
  \[d_\cyl(u(s,t), \ol \gamma(s,t)) \leq c_u e^{-\delta s},\]
  finishing the proof of the
  Theorem \ref{thm:main}.

  In the case of a horizontally tamed almost complex structure, we need to prove
  the weaker statement that the decay estimate \eqref{eq:decay2cyl}
  holds for some $\delta$, as opposed to any
  $\delta < \lam_T$. The proof of the exponential
  convergence result (Proposition \ref{prop:expdecay}) carries over to
  this case with the difference that the application of the refined
  isoperimetric inequality (Proposition \ref{prop:isop-ref}) is
  replaced by the ordinary isoperimetric inequality (Proposition
  \ref{prop:iso}). The rest of the proof is identical.
\end{proof}

\section{A monotonicity result for strips}\label{sec:stripmonot}

In this section, we state and prove an analogue of Gromov's
monotonicity result for pseudoholomorphic maps on strips. The result applies
to neighborhoods of intersections of Lagrangians, without requiring
that the image of the map contain a Lagrangian intersection point. The
domains of the maps occurring in this result are subsets $C$ of the
strip $\R \times [0,1]$. In particular, we assume that
$C \subset \R \times [0,1]$ is a manifold with corners whose
codimension one boundary is a union of two kinds of strata, namely the
{\em strip boundary} and the {\em non-strip boundary} defined as
\[\partial_{\st}C:= C \cap (\R \times \{0,1\}), \quad \partial_\ns C:=\ol{\partial C \cap (\R \times (0,1))}\]
respectively,
and the codimension two boundary consists of points in $\partial_\st C \cap \partial_\ns C$.

\begin{theorem}\label{thm:monot}
  {\rm(Monotonicity for strips)}
  Let $L_0, L_1$ be cleanly intersecting Lagrangian submanifolds in a
  symplectic manifold $(\X,\om)$ equipped with a Riemannian metric $g$, and let $U_\J$ be a
  $C^0$-neighborhood on the space of tamed almost complex structures.
  Let $p \in L_0 \cap L_1$.  There are constants $R>0$,
  $c_\monot>0$ such that
  \begin{enumerate}
  \item for any point $x$ and $r>0$ such that $B_r(x) \subset B_R(p)$,
  \item a compact surface with corners $C \subset \R \times [0,1]$,
  \item and a map $u: C \to \X$ that is pseudoholomorphic with respect to a
    domain-dependent almost complex structure $J:C \to U_\J$, and which
    satisfies the Lagrangian boundary conditions
    \[u(C \cap (\R \times \{i\})) \subset L_i \quad \text{for
        $i=0,1$ and} \quad u(\partial_\ns C) \subset \partial
      B_r(x), \]
    and $x \in u(C)$,
  \end{enumerate}
  the symplectic area of $u$ is bounded as
  \[\int_C u^*\om \geq c_\monot r^2. \]
\end{theorem}

For pseudoholomorphic curves, the symplectic area is equal to metric area.
For use in the proof of the result, we recall 
the notion of metric area
from \cite[Section 2.2]{ms:jh}.
For a tamed almost complex structure $J$ on a symplectic manifold $(\X,\om)$, and a map $u: C \to \X$
on a Riemann surface $C$, the metric area is 
\begin{equation}
  \label{eq:enexp1}
 E_\om(u,C):= \thh \int_C |du|_J^2 \om_C,
\end{equation}
where $\om_C \in \Om^2(C)$ is an area form on $C$, and for any $z \in C$,
\begin{equation}
  \label{eq:enexp2}
  |du_z|_J^2:=\tfrac 1 {|v|^2_C}(|du(v)|_J^2 + |du(j_C v)|_J^2), 
\end{equation}
where $v \in T_z C$ is a non-zero vector, and 
in the right hand side $|\cdot|_J$ is the norm  $(\om(\cdot, J\cdot))^{1/2}$
on $T\X$, 
and $|\cdot|_C$ is the norm $\om_C(\cdot, j_C \cdot)$ on $TC$.
The quantity $|du|_J^2 \om_C$ is independent of $\om_C$ and $v$. Given conformal coordinates  $(s+it)$ on $C$, we have 
\[E_\om(u,C)=\thh \int_C(|\partial_s u|_J^2 + |\partial_t u|_J^2) ds \wedge dt.  \]
The almost complex structure $J$ is allowed to be domain-dependent, that is, $J$ is a map from $C$ to the space of tamed almost complex structures, in which case, the metric $|\cdot|_J$ is also domain-dependent. If the map $u: C \to M$ is $J$-holomorphic, then
\begin{equation}
  \label{eq:enexp}
  \int_C u^*\om=E_\om(u,C).
\end{equation}
The following is frequently used.

\begin{observation}{\rm(Uniform equivalence of metrics)}
  Let $(\X,\om)$ be a compact symplectic manifold with a Riemannian metric $g$,
  $C$ is a Riemann surface,
  $U_\J$ is a $C^0$-neighborhood in the space of $\om$-tamed almost complex structures on $\X$, 
  and $J: C \to U_\J$ is a domain-dependent tamed almost complex structure. Then,
  there is a constant $c>0$ such that 
  for any $z \in C$,  $v \in T_{u(z)}\X$,
  \begin{equation}
    \label{eq:eqmet}
    c^{-1}|v|_g \leq |v|_{J(z)} \leq c|v|_g.  
  \end{equation}
\end{observation}

\subsection{Isoperimetric inequality}
We give an isoperimetric inequality on a symplectic manifold $(\X,\om)$ equipped
with a pair of cleanly intersecting Lagrangian submanifolds $L_0$, $L_1$. Given a
short path whose end-points lie on the Lagrangians, the isoperimetric
inequality bounds the `action' of the path, which is defined below.
For our application, it suffices to focus on paths that lie in fixed Darboux neighborhoods.

\begin{definition}\label{def:darb}
  {\rm (Adapted Darboux chart)}
  Let $L_0, L_1 \subset (\X,\om)$ be a pair of cleanly intersecting Lagrangian submanifolds and $p \in L_0 \cap L_1$. 
A Darboux chart 
    \[\Phi : U \to V \subset (\R^{2n}, \om_0=\sum_i dx_i \wedge dy_i)), \quad \Phi(p)=0, \quad \phi^*\om_0=\om  \]
    on $U \subset \X$ is {\em adapted to $L_0$, $L_1$} if there are linear subspaces $\Lam_0, \Lam_1 \subset \R^{2n}$ such that 
    \begin{equation}
      \label{eq:lageuc}
      \Phi(L_j \cap U)=\Lam_j \cap V, \quad j=0,1.    
    \end{equation}
\end{definition}

\begin{remark}
  In a Darboux chart as above, $\om$ has a primitive $\lam$ that
  vanishes on $L_0$, $L_1$. Indeed, the linear coordinates on
  $\R^{2n}$ may be chosen so that
  \[\Lam_0=\{x_1=\dots=x_n=0\}, \quad \Lam_1=
    \{x_1=\dots=x_m=y_{m+1}=\dots=y_n=0\},\]
  and the primitive on $U$ is $\lam:=\Phi^*(\sum_i x_i dy_i)$.
\end{remark}

\begin{proposition}{\rm(Pozniak \cite[Proposition 3.4.1]{Poz:thesis})}
  \label{prop:poz}
  Suppose $L_0$, $L_1$ are cleanly intersecting Lagrangian
  submanifolds in a symplectic manifold $(\X,\om)$.  Any
  $p \in L_0 \cap L_1$ has a Darboux chart adapted to $L_0$, $L_1$.
\end{proposition}

\noindent The result  \cite[Proposition 3.4.1]{Poz:thesis} requires the
intersection of the Lagrangians to be compact.  Our result does
not require compactness, because it is a result about the neighborhood
of a single point in $L_0 \cap L_1$, in contrast to Pozniak's result
which is a neighborhood theorem for all of $L_0 \cap L_1$.

We define an action functional on paths lying in an adapted Darboux neighborhood.

 \begin{definition} {\rm(Action functional)}
   Let $L_0$, $L_1$ be a pair of cleanly intersecting Lagrangian
   submanifolds in a symplectic manifold $(\X^{2n},\om)$.  Let
   $U \subset \X$ be a Darboux ball  adapted to $L_0$, $L_1$ as in Definition \ref{def:darb}, and let $\lam \in \Om^1(U)$ be a primitive of $\om$ that vanishes on $L_0$, $L_1$.  Let
   $\gamma : \Gamma \to U$ be a map defined on a compact oriented $1$-manifold $\Gamma$ with boundary, (that is, $\Gamma$ is a finite collection of loops and paths) satisfying
   \[\gamma(\partial \Gamma) \subset L_0 \cup L_1.\]
   The {\em action} of $\gamma$ is
   \[\A_\loc(\gamma):=\int_\Gamma \gamma^*\lam. \]
\end{definition}

The action is equal to the $\om$-area of a (possibly disconnected) surface in $U$ whose
non-strip boundary is $\gamma$: Suppose $\gamma: \Gamma \to U$ is a
path and $u_\gamma: C \to U$ is a map on a subset
$C \subset \R \times [0,1]$ of a strip with $\partial_{\ns}C=\Gamma$
with the same orientation, and satisfying
\[u(C \cap (\R \times \{i\})) \subset L_i \quad \text{for }i=0, 1, \quad u_\gamma|\partial_\ns C \equiv \gamma.\]
Then, by Stokes theorem, and the fact that the primitive $\lam$ vanishes on $L_0 \cup L_1$, we have 
\[\int_Cu_\gamma^*\om = \A_\loc(\gamma). \]
 The action $\A_\loc$ does not depend on the primitive $\lam$ or the map $u_\gamma$ as long as the image of $u_\gamma$ lies in $U$. For a small enough $U$,
\footnote{For the action to be independent of $U$, $U$ needs to be contained in a geodesic ball at $p$ whose radius is half the injectivity radius, with respect to a metric for which $L_0$, $L_1$ are totally geodesic, see \cite[Section 4.4]{ms:jh}.}
 $\A_\loc$ is independent of $U$.

\begin{proposition}{\rm(Isoperimetric inequality)}
  \label{prop:iso}
  Let $(\X,\om)$ be a symplectic manifold with cleanly intersecting Lagrangians $L_0$, $L_1$
  and a metric $g$, and let $U \subset \X$ be a bounded Darboux ball centered at $p$ with coordinates adapted to $L_0$. $L_1$.
  Then, there is a constant $c>0$ such that for any path $\gamma : \Gamma \to U$
defined on a compact oriented (possibly disconnected) one-manifold $\Gamma$ with $u(\partial \Gamma) \subset L_0 \cup L_1$,
  \begin{equation}
    \label{eq:isop}
    |\A_\loc(\gamma)| \leq c \ell_g(\gamma)^2  
  \end{equation}
  where $\ell_g(\gamma)$ is the length of $\gamma$ with respect to the metric $g$.
\end{proposition}
\begin{proof}
  We prove the result assuming that $\Gamma$ is connected, since the general result follows by adding the inequalities corresponding to each of the connected components.
We focus on the case when $\gamma$ is a path
  with end-points in $L_0$ and $L_1$, referring the reader to \cite[Remark 4.4.3]{ms:jh}
  for other cases. 
  We may assume
 that $\gamma$ maps to $\R^{2n}$ with the standard
  symplectic form and Euclidean metric, $p$ is the origin, 
  and $L_0$, $L_1$ are linear subspaces as in \eqref{eq:lageuc},
  since the Euclidean metric and $g$ differ by a factor that has an
  upper and lower bound on $U$ (see \eqref{eq:eqmet}).
The length of a path $\gamma$ with respect to the Euclidean metric is denoted by $\ell(\gamma)$. 
  We consider a sector of a disk bounding $\gamma$, namely a smooth map
  \[u_\gamma : C:=\{re^{i\theta}: r \in [0,1], \theta \in [0,1]\} \to U, \quad u_\gamma(re^{i\theta})=r\gamma(\theta),\]
  where we note that $u(\{\theta=0\}) \subset L_0$, $u(\{\theta=1\}) \subset L_1$.
 There is a constant $c>0$ independent of $u$ such that 
  \begin{multline*}
    |\A_\loc(\gamma)|=\left|\int u_\gamma^*\om_0\right| \leq \int_C|\om_0(\partial_r u_\gamma, \partial_\theta u_\gamma)| dr d\theta \leq  \int_C|\partial_r u_\gamma|\cdot |\partial_\theta u_\gamma|  dr d\theta \\
    \leq \thh \int_0^1|\gamma(\theta)|\cdot |\gamma'(\theta)| d\theta \leq c \ell(\gamma)^2,
  \end{multline*}
  where in the last inequality, we use the fact that 
  $|\gamma(\theta)| \leq c|\ell(\gamma)|$, which is seen as follows: 
Let $\rho=\min_\theta |\gamma(\theta)|$. Then for all $\theta$,  $|\gamma(\theta)| \leq \rho + \ell(\gamma)$.
Finally,
$\rho \leq  \frac 1 {\sqrt{2}}\ell(\gamma)$, because $\gamma$ lies in the set $\R^{2n} \bs (B_\rho(L_0 \cap L_1))$, where the distance between and $L_0$ and $L_1$ is $\sqrt{2} \rho$. 
\end{proof}
\subsection{Proof of the strip monotonicity result}
\begin{proof}
  [Proof of Theorem \ref{thm:monot}]
  We fix the constant $R$ so that $B_R(p)$ is contained in a Darboux ball adapted to the Lagrangian submanifolds $L_0$, $L_1$ from Proposition \ref{prop:poz}.

  We first define a function $E(r)$ that is used in the proof. 
  Suppose $B_{r_0}(x) \subset B_R(p)$ and the map $u:C \to X$ are as in the hypothesis so that
  \[u(C) \subset \ol B_{r_0}(x), \quad u(\partial_\ns C) \subset\partial B_{r_0}(x).\]
  Define
  \[d_x:C \to [0,r_0], \quad d_x(z):=d(x,u(z)),\]
  where $d$ is the distance on $X$ with respect to the metric $g$.
  For any $0 \leq r \leq r_0$, we say that the level set $\{d_x=r\} \subset C$ is {\em regular} if
  $r$ is a regular value for both $d_x$ and $d_x|\partial_\st(C)$, so that $\{d_x=r\}$ is a compact one-manifold that intersects the strip boundary transversely. 
  By a small perturbation of the metric $g$, 
  we may assume that all but a finite number of level sets of $d_x$ are regular. 
Define
  \[C_r:=\{d_x \leq r\}=u^{-1}(B_r(x)), \quad E(r):=\int_{C_r}u^*\om.\]

  We bound the derivative of $E(r)$ by the length of the non-strip boundary of $C_r$. Note that for a regular level set $\{d_x=r\}$ is equal to
  the non-strip boundary $\partial_\ns C_r$. 
  In the rest of this proof $c$ denotes a $u$-independent constant, whose value changes across instances. 
  \begin{claim}
    There is a $u$-independent constant $c>0$ such that for a regular level set $\{d_x=r\}$, 
  \begin{equation}
    \label{eq:ddrarea}
    \ddr E(r) \geq c \ell_g(\partial_\ns C_r),    
  \end{equation}
  where $\ell_g$ is length with respect to the metric $g$.
  \end{claim}
  \begin{proof}[Proof of Claim]
    To start with, we restrict ourselves to the case that $\{d_x=r\}$ is a path, and we choose suitable coordinates on the domain.
    Let $U \subset C$ be a neighborhood of
    $\partial_\ns C_{r}$ that does not contain any critical points of $d_x$. We claim that 
    there are coordinates $(\rho,\tau)$ on $U$ such that $\rho:=d_x$ and $\tau$ satisfies 
    \[d\tau = f(\rho,\tau) d\rho \circ j_C\]
    for some $f:U \to \R_+$ for which $f \equiv 1$ on
    $\{\rho=r\}$.
The coordinate $\tau:U \to \R$ is explicitly defined (up to an additive constant) as follows: 
On $\{d_x=r\}$, define $\tau$ by the
    condition $d\tau=d\rho \circ j_C$. Extend $\tau$ to $U$ by the condition that  $\tau$ is constant on the integral curves
    of the line field $\ker(d\rho \circ j_C)$.  

We differentiate $E(r)$ using the coordinates chosen above. 
    Let $r_1<r$ be such that $C_r \bs C_{r_1} \subset U$.
    By the pseudoholomorphicity of
    $u$, and \eqref{eq:enexp}, \eqref{eq:enexp2}, for any $r_1 \leq \rho \leq r$,
    \[E(\rho) - E(r_1)=E_\om(u, C_\rho \bs C_{r_1})=\thh \int_{C_\rho \bs C_{r_1}}
      (\tfrac 1 f |\tfrac{\partial u}{\partial \rho}|^2_{_J} +
      f|\tfrac{\partial u}{\partial \tau}|^2_J)d\rho \wedge
      d\tau. 
    \]
    The above expression may be obtained by taking, for example, $\om_C:=d\rho \wedge d\tau$
    and $v:=\pprho$ in \eqref{eq:enexp1},  \eqref{eq:enexp2} and observing that $j_C v= f \pptau$.  
%     where $|\cdot|_J$ is the domain-dependent metric
%     $\hh(\om(\cdot, J\cdot) + \om(\cdot, J\cdot))$.  Since $J$ takes
%     values in a $C^0$-neighborhood, and the image of $u$ lies in the
%     compact set $\ol {B_R(p)}$, there is a constant $c$ such that for
%     any $z \in C$ and any $v \in T_{u(z)}C$,
% %
%     \begin{equation}
%       \label{eq:gjeq}
%       |v|_{J(z)} \geq c|v|_g.  
%     \end{equation}
%
    Therefore,
    \begin{multline*}
      \ddrho E(\rho)=\thh \int_{\partial_\ns C_\rho}(|\partial_\rho
      u|^2_{_J} + |\partial_\tau u|^2_J)d\tau
      \geq \int_{\partial_\ns C_{r_0}}(|\partial_\rho u|_{_J} |\partial_\tau u|_J)d\tau  \\
      \geq c\int_{\partial_\ns C_{r_0}}(|\partial_\rho u|_{_g}
      |\partial_\tau u|_g)d\tau \geq c\int_{\partial_\ns
        C_{r_0}}|\partial_\tau u|_gd\tau = c\ell_g(\partial_\ns
      C_{r_0}).
    \end{multline*}
    The second to last estimate follows from the uniform equivalence \eqref{eq:eqmet} of the $g$ and $J$ metrics, which holds because the image of $u$ lies in the compact set $\ol{B_R(p)}$. 
    The last estimate follows from $|\partial_\rho u|_g \geq 1$, which is
    a consequence of the definition of the coordinate $\rho$. Thus the result is proved in the case that $\{d_x=r\}$ is a path.

    In the case when $\{d_x=r\}$ is a loop, the same proof applies
    with the difference that the coordinate $\tau$ is
    $\R/\kappa \Z$-valued for a suitable $\kappa>0$. In the case when
    $\{d_x=r\}$ has multiple components, we may prove the result on
    each of the components, and add the inequalities.
  \end{proof}

  We now finish the proof of Theorem \ref{thm:monot} by using the isoperimetric
  inequality.  Recall that $u$ maps to a Darboux neighborhood where
  $\om$ has a primitive $\lam \in \Om^1(\X)$ that vanishes on $L_0$,
  $L_1$. Therefore, for any regular level $r$,
  \begin{equation}
    \label{eq:isop-cr}
  E(r)=\int_{C_r}u^*\om = \int_{\partial_\ns C_r} u^*\lam =\A_\loc(\partial_\ns C_r) \leq c 
    \ell_g(\partial_\ns C_r)^2.  
  \end{equation}
  The last inequality is a consequence of
  the isoperimetric inequality \eqref{eq:isop} applied to 
  the non-strip boundary $\partial_\ns C_r$, which is a collection of paths and loops. 
  By \eqref{eq:isop-cr} and \eqref{eq:ddrarea}, $\sqrt{E(r)} \leq c\ddr E(r)$ for almost every $r$, and consequently, $\ddr \sqrt{E(r)} \geq c$. 
  Integrating this inequality,  we obtain $E(r_0) \geq cr_0^2$.
\end{proof}

\section{Exponential convergence}\label{sec:expconv}
In this section, we prove that a holomorphic strip whose boundary lies
on a pair of cleanly intersecting Lagrangians, converges to a point at
an exponential rate.  We also show that the rate of exponential
convergence can be chosen to be any number that is bounded by the smallest
non-zero eigenvalue of an asymptotic operator. Readers wishing to skip the
proof of this second statement may jump ahead to Proposition
\ref{prop:expdecay}, the simplification in the proof is indicated in
Remark \ref{rem:anydelta}.  The results in this section are
 adapted from Robbin-Salamon \cite{RS:strips} and Schm\"aschke \cite{Schm:clean}.

\subsection{Asymptotic operator}
\label{subsec:asym}
The rate of exponential convergence in Theorem \ref{thm:main} is
bounded by a constant $\lam_T$ arising from an asymptotic operator.
In this section, we define the asymptotic operator  and the constant $\lam_T$,
and prove a refined version of the isoperimetric inequality involving this constant.

We describe the asymptotic operator for pseudoholomorphic strips, the
version for pseudoholomorphic cylinders is later discussed in Remark
\ref{rem:cylop}.   Let $L_0$, $L_1$ be a
pair of cleanly intersecting Lagrangian submanifolds in a symplectic
manifold $(X,\om)$. We consider strips
$u: \R_{\geq 0} \times [0,1] \to X$ satisfying
\[u(\R_{\geq 0} \times \{0\}) \subset L_0, \quad  u(\R_{\geq 0} \times \{1\}) \subset L_1\]
that are pseudoholomorphic with respect to a family of compatible almost complex structures
$\J:=\{J_t\}_{t \in [0,1]}$, and for which $\lim_{s \to \infty} u(s,t)=p$ for some $p \in L_0 \cap L_1$. 
The {\em asymptotic operator} at $p \in L_0 \cap L_1$ is defined as 
\begin{equation}
  \label{eq:apJ}
  A_{p,\J} := A_{p,\J}(X, L_0, L_1): \Gamma_L([0,1], T_pX) \to \Gamma([0,1], T_pX), \quad v \mapsto J_t \partial_t v,  
\end{equation}
where
\[\Gamma_L([0,1], T_pX):=\{v:[0,1] \to T_pX:v(0) \in T_pL_0, v(1) \in T_p L_1\}.\]
The name `asymptotic operator' arises from the fact that the
linearized Cauchy-Riemann operator of the pseudoholomorphic strip $u$
is of the form $\pps + A(s)$, and $A(s)$ converges to $A_{p,\J}$ as
$s \to \infty$.

The operator $A_{p,\J}$ is the Hessian of the action
functional in Lagrangian Floer theory (see \cite{Floer:lag}), whose critical points are
intersections of a pair of Lagrangians.  The Morse case is when the pair of
Lagrangians intersect transversely, and in this case, for any
$p \in L_0 \cap L_1$, the operator $A_{p,\J}$ is invertible. The
Morse-Bott case is when the pair of Lagrangians intersect cleanly, and in
this case for any $p \in L_0 \cap L_1$, $0$ lies in the spectrum of
$A_{p,\J}$, but the spectral gap $\lam_p$ at $0$ (see \eqref{eq:lamp+})
is positive. For a pseudoholomorphic strip converging to $p$, the rate
of exponential convergence is arbitrarily close to $\lam_p$.

The operator $A_{p,\J}$ extends to an unbounded operator $L^2([0,1], T_p X) \to L^2([0,1], T_pX)$
with domain
\[\D(A_{p,\J})=H^1_L([0,1], T_p X):=\{v \in H^1([0,1], T_pX): v(0) \in T_pL_0, v(1) \in T_pL_1\}.\]
Here, the inner product on $L^2([0,1], T_p X)$ is given by
  \begin{equation}
    \label{eq:L2}
    \bran{v_1,v_2}_{L^2}:=\int_{t \in [0,1]}\om(v_1(t), J_t v_2(t)) dt, \quad v_1, v_2:[0,1] \to T_pM.    
  \end{equation}
The {\em spectrum} $\sig(A_{p,\J})$ is the set of values $\lam \in \C$ for which $(A_{p,\J}-\lam):H^1_L \to L^2$ is not invertible.
The quantity
  \begin{equation}
      \label{eq:lampdef}
      \lam_p:=\inf\{|\lam|: \lam \in \sig(A_{p,\J})\bs \{0\}\}    
    \end{equation}
    is called the {\em spectral gap of $A_{p,\J}$ at $0$}. 
The following are some spectral properties of the asymptotic operator on strips.

\begin{lemma}
  {\rm(Spectral properties of the asymptotic operator)}\label{lem:str}
  Let $L_0, L_1$ be a pair of cleanly intersecting Lagrangian submanifolds in a symplectic manifold $(X,\om)$, and let $\J=\{J_t\}_{t \in [0,1]}$ be a family of $\om$-compatible almost complex structures. 
  \begin{enumerate}
  \item \label{part:str1} For any
    $p \in L_0 \cap L_1$, the spectral gap of $A_{p,\J}$ at $0$ is positive, that is, 
    \begin{equation}
      \label{eq:lamp+}
      \lam_p>0.
    \end{equation}
  \item  \label{part:str2}
The space $L^2([0,1], T_pX)$ has an orthonormal basis consisting of eigensections of $A_{p,\J}$. 
  \item \label{part:str3} The map $L_0 \cap L_1 \ni p \mapsto \lam_p \in \R_+$ is continuous.
  \end{enumerate}
\end{lemma}
\begin{proof}
  As an unbounded operator between (real) Hilbert spaces, $A_{p,\J}: L^2 \to L^2$ is self-adjoint, Indeed, for $v_1, v_2 \in H^1_L$, abbreviating $A_p:=A_{p,\J}$, 
  \[\bran{v_1, A_p v_2}_{L^2}=-\int_0^1 \om(v_1, \partial_t v_2)=\int_0^1 \om(\partial_t v_1, v_2)=\bran{A_p v_1, v_2},\]
  where the middle equality is obtained by integrating by parts, since
  $\om(v_1(j),v_2(j))=0$ for $j=0,1$.\footnote{The proof of self-adjointness requires $J_t$ to be $\om$-compatible, and not just $\om$-tamed.} 
 
 The results now follow from the conclusions of \cite[Proposition
 3.12]{wendl:sft}.  The spectrum is a discrete subset of $\R$, and
 each element is an eigenvalue of finite multiplicity.  Therefore
 \eqref{part:str1} follows.  For any $\lam \in \R$ not in the
 spectrum, $(A_p - \lam)$ is invertible. Composing with the compact
 inclusion $H^1 \hra L^2$, the operator $(A_p-\lam)^{-1}:L^2 \to L^2$
 is compact and self-adjoint. By the spectral theorem, $L^2$ has an
 eigenbasis of $(A_p-\lam)^{-1}$, which is also an eigenbasis of
 $A_p$, proving \eqref{part:str2}.  For any bounded interval
 $[a,b] \subset \R$, whose end-points are not eigenvalues of $A_p$,
 there is a neighborhood $U \subset L_0 \cap L_1$ of $p$, such that the
 number of eigenvalues of $A_q$ in $[a,b]$, counted with multiplicity,
 is equal for all $q \in U$. Now, \eqref{part:str3} follows from the
 fact that the multiplicity of the eigenvalue $0$ for $A_q$ is equal
 to $\dim(L_0 \cap L_1)$ for all $q$, since $\ker(A_q)$ consists of
 constant sections of $\Gamma([0,1], T_qM)$ satisfying the boundary
 conditions.
\end{proof}

The following is a refined version of the isoperimetric inequality
(Proposition \ref{prop:iso}) involving the spectral gap at
$0$ of the asymptotic operator.  For a path $\gamma:[0,1] \to (X,\om)$,
and a family $\J=\{J_t\}_{t \in [0,1]}$ of $\om$-compatible almost complex structures, 
define
\[\ell_\J(\gamma):=\left(\int_0^1|\partial_t \gamma|^2_{J_t} dt\right)^{1/2}.\]
We note that the definition of $\ell_\J$ differs from standard definitions of length in that it uses an $L^2$-norm.

\begin{proposition}
  {\rm(Isoperimetric inequality, refined)}
  \label{prop:isop-ref}
   Let $L_0, L_1$ be a pair of cleanly intersecting Lagrangian submanifolds in a symplectic manifold $(X,\om)$. 
   Let $p \in L_0 \cap L_1$, let $\J=\{J_t\}_{t \in [0,1]}$ be a family of $\om$-compatible almost complex structures, 
   and let $\lam_p>0$ be
  as in \eqref{eq:lampdef}. For any $0<\delta< \lam_p$, there is a neighborhood $U_\delta \subset X$ of $p$ such that for any path
  \[\gamma:[0,1] \to U_\delta\]
  with $\gamma(0) \in L_0$, $\gamma(1) \in L_1$,
  \[2\delta |\A_{\loc}(\gamma)| \leq  \ell_\J(\gamma)^2.\]
\end{proposition}
\begin{proof}
  The asymptotic operator is conjugate to an operator on $\R^{2n}$ via
  a Darboux chart as follows: We choose a Darboux chart centered at
  $p$ adapted to $L_0$, $L_1$, that is, $\Phi:U \to (\R^{2n}, \om_0)$
  is a symplectomorphic embedding satisfying
  \[\Phi(p)=0, \quad \Phinv(\Lam_0)=L_0 \cap U, \quad \Phinv(\Lam_1)=L_1 \cap U,\]
  where $\Lam_0, \Lam_1 \subset \R^{2n}$ are linear subspaces. For any
  $q \in U$, $t \in [0,1]$, we view $J_t(q)$ as a linear complex
  structure on $\R^{2n}$ via the map $d\Phi_q$. The operator
  \[A_0 : \Gamma_\Lam([0,1],\R^{2n}) \to \Gamma([0,1],\R^{2n}), \quad v \mapsto J_t(p)\partial_tv, \]
with domain 
  \[\Gamma_\Lam([0,1],\R^{2n}):=\{v:[0,1] \to \R^{2n}: v(0) \in \Lam_0, v(1) \in \Lam_1\},\]
  is conjugate to
  $A_{p,\J}$ from \eqref{eq:apJ} via the
  identification $d\Phi_p : T_pM \to \R^{2n}$.

  The action functional can be rewritten in terms of the asymptotic operator:
  Define $\xi:=\Phi \circ \gamma : [0,1] \to \R^{2n}$. Let $u:[0,1]^2 \to \R^{2n}$ be $u(s,t)=s\xi(t)$. 
  Since $\Phi$ is a symplectomorphism, $\A_{\on{loc}}(\gamma)=|\int_{[0,1]^2}u^*\om_0|$, and 
  we have 
  \begin{multline*}
    %\A_{\on{loc}}(\gamma)=
    \int_{[0,1]^2}u^*\om_0=\int_0^1 \int_0^1 \om_0(\partial_su, \partial_t u) ds dt= \hh \int_0^1 \om_0(\xi, \partial_t \xi) dt = - \thh \bran{\xi, J_t(p) \partial_t \xi}_{L^2},
  \end{multline*}
  where the last equality is 
   by the definition \eqref{eq:L2} of the $L^2$-inner product. Therefore, 
  \begin{equation}
    \label{eq:aaloc}
    |\A_{\on{loc}}(\gamma)| =\thh |\bran{\xi, A_0 \xi}_{L^2}|.   
  \end{equation}

  Next, we bound the length of a path using the asymptotic operator. 
  We 
  observe that for any $\eps>0$ there is a neighborhood $U_\eps \subset M$ of $p$ such that
  for all $q \in U_\eps$, $v \in \R^{2n}$, $t \in [0,1]$, 
  \[(1-\eps)\om_0(v,J_t(q)v) \leq \om_0(v, J_t(p)v);\]
  a similar estimate is proved as part of the proof of Proposition 4.4.1 in \cite{ms:jh}. 
Assuming that the path $\gamma$ lies in the neighborhood $U_\eps$ of $p$,  
  the length of the path $\gamma$ is 
  \begin{multline}\label{eq:ellJ}
    \ell_\J(\gamma)^2=\int_0^1\om_0(\xi'(t), J_t(\gamma(t))\xi'(t))dt \\
    \geq (1-\eps)\int_0^1\om_0(\xi'(t), J_t(p)\xi'(t))dt = (1-\eps)\bran{A_0\xi, A_0\xi}_{L^2}. 
  \end{multline}

  By the orthonormal eigen-decomposition of $A_0$ (Lemma \ref{lem:str} \eqref{part:str2}), we have
  \begin{equation}
    \label{eq:lamp-bd}
     \lam_p |\bran{\xi, A_0\xi}| \leq \bran{A_0 \xi, A_0\xi}.
  \end{equation}
  Given $0<\delta<\lam_p$, choose $\eps:=1-\frac {\delta} {\lam_p}$. The
  result follows by combining \eqref{eq:aaloc}, \eqref{eq:ellJ} and
  \eqref{eq:lamp-bd}.
\end{proof}

\begin{remark}
  \label{rem:cylop}
  There is an asymptotic operator for pseudoholomorphic cylinders
  analogous to the one on strips, which is defined as follows.  For a
  Reeb orbit $\gamma: \R/T\Z \to M$ of period $T>0$ in a contact manifold $M$, and a
  cylindrical almost complex structure $J$ on the symplectization
  $W:=\R \times M$, the {\em asymptotic operator} is
  \[A_{\gamma,J}:=A_{\gamma,J}(M) : \Gamma(\R/\Z,\gamma_T^*TW) \to
    \Gamma(\R/\Z,\gamma_T^*TW), \quad v \mapsto JL_{TX} v,\]
  where $\gamma_T:\R/\Z \to W$ is defined as
  $\gamma_T:=(0,\gamma(T \cdot))$, $X$ is the Reeb vector field and
  $L_{TX}$ is the Lie derivative. Via the twisting and doubling operation
  used in the proof of Theorem \ref{thm:main}, the operator
  $A_{\gamma,J}$ translates to an asymptotic operator $A_{p,\J}$ on
  strips.  The operator $A_{p,\J}$ is defined on the manifold $W^2$
  with the pair of Lagrangians $\Delta$, $\Delta_T$ and the
  intersection point
  $p:=((0,\gamma(0)), (0,\gamma(0))) \in \Delta \cap \Delta_T$, and
  $\J=(J_t)_{t \in [0,\hh]}$ is as in \eqref{eq:Jtdef}. The operators
  $A_{\gamma,J}$, $A_{p, \J}$ are conjugate via the map
  \begin{equation*}
    % \begin{multline*}
    \Gamma(\R/\Z, \gamma_T^*TW) \to \Gamma([0,\thh], T_pW^2),  \quad 
    v \mapsto  (t \mapsto (\psi_{Tt}^*v_t, \psi_{T(1-t)}^* v_{(1-t)})).
    % \end{multline*}
  \end{equation*}
\end{remark}
\begin{remark}
  In this remark, we will define the constant $\lam_T$ in Theorem
  \ref{thm:main} as the minimum spectral gap at zero for the
  asymptotic operators corresponding to Reeb orbits of period $T$.
    Using the terminology of
  Theorem \ref{thm:main}, define 
  \begin{multline}
    \label{eq:gamt-cyl}
    \lam_T:=\inf\{|\lam|: \lam \in \sig(A_{\gamma,J}(M)) \bs \{0\}, \\
    \gamma : \R/T\Z \to M \text{ is a $T$-periodic Reeb orbit}\}.
  \end{multline}
The constant may equivalently be defined via strips as  
  \begin{equation}
    \label{eq:gamt-str}
    \lam_T:=\inf\{|\lam|: \lam \in \sig(A_{p,\J}(W^2, \Delta, \Delta_T)) \bs \{0\}, p \in \Delta \cap \Delta_T\},   
  \end{equation}
  where $W$, $\J$, $p$ are obtained from $M$, $J$, $\gamma$ is as in Remark \ref{rem:cylop}. We observe that
  \[\lam_T=\inf\{ \lam_p:p \in \Delta_0 \cap \Delta_T\},\]
  where $\lam_p$ is the spectral gap at $0$ for the asymptotic operator $A_{p,\J}$. 
  The constant $\lam_T$ is positive, 
  \[\lam_T>0,\]
  because $\Delta \cap \Delta_T \simeq \R \times N_T$, the map
  \[\Delta \cap \Delta_T \to \R_{>0}, \quad p \mapsto \lam_p\]
  is $\R$-invariant and continuous (Lemma \ref{lem:str}\eqref{part:str3}), and $N_T$ is compact.
\end{remark}

\subsection{Exponential convergence for strips}
We continue with the setting of a pair of cleanly intersecting
Lagrangians and prove a result that is an application of the
isoperimetric inequality.  We show that
if a pseudoholomorphic strip uniformly converges to a point at
infinity, then the rate of convergence is in fact exponential. The
techniques are adapted from \cite[Theorem 4.1.2]{ms:jh}.

\begin{proposition}{\rm(Exponential convergence)}
  \label{prop:expdecay} Let $L_0$, $L_1$ be a pair of cleanly
  intersecting Lagrangians in a symplectic manifold $(\X,\om)$.
  % , $L_0$,
  % $L_1$, $g$ and the Darboux ball $U \subset \X$ centered at
  % $p \in L_0 \cap L_1$ be as in the hypothesis of Proposition
  % \ref{prop:iso}.
  Let $u: C \to \X$ be a map on
  $C:=\R_{\geq 0} \times [0,1]$ that is pseudoholomorphic with respect
  to a $[0,1]$-dependent $\om$-compatible almost complex structure
  $\J=(J_t)_{t \in [0,1]}$, which satisfies the Lagrangian boundary
  conditions
  \[u(C \cap (\{i\} \times \R)) \subset L_i \quad \text{for
      $i=0,1$}, \]
  and whose $\om$-area $\int_C u^*\om$ is finite.  Suppose further
  that the maps $u(s,\cdot) : [0,1] \to U$ converge uniformly to a
  constant map mapping to $p$ as $s \to \infty$ for some 
  $p \in L_0 \cap L_1$.  Then, for any $0<\delta<\lam_p$ (defined in
  \eqref{eq:lampdef}), there is a $c_u>0$ such that for all $s \geq 0$,
  \begin{equation}
    \label{eq:decayref}
    d(u(s,t),p) \leq c_u e^{-\delta s},  
  \end{equation}
  where $d$ is the distance on $\X$ with respect to a Riemannian
  metric $g$ on $X$.
\end{proposition}
\begin{proof}%[Proof of Proposition \ref{prop:expdecay}]
  First, we show that the $\om$-area of the strip decays exponentially
  towards the end. Denote the $\om$-area on the truncated strip by
  \[E(s_0):=\thh \int_{C \cap \{s \geq s_0\}} u^*\om 
    = \int_{s_0}^\infty \int_0^1 |\partial_t u|^2_{J_t}  dt ds,
  \]
  where the second equality follows from \eqref{eq:enexp}, and 
  since $\partial_su=-J_t \partial_t u$,  $|\partial_t u|_{J_t}=|J_t\partial_t u|_{J_t}$.
  For any $s \geq 0$, let
  \[\ell_\J(s):=\left(\int_0^1 |\partial_tu(s,t)|^2_{J_t} dt\right)^{1/2},\]
  denote the $\ell_J$-length of the path $u(s,\cdot)$. Then,
  \begin{equation}
    \label{eq:c0}
    -\dds E(s)  = \ell_\J(s)^2.
  \end{equation}
  Since the paths $u(s,\cdot)$ converge to $p \in L_0 \cap L_1$ as $s \to \infty$, the
  refined version of the isoperimetric inequality (Proposition
  \ref{prop:isop-ref}) is applicable on $u(s,\cdot)$ if $s$ is large enough. That is, 
  for any $0<\delta <\lam_p$, Proposition \ref{prop:isop-ref} implies that
there exists $s_0$ such that for $s \geq s_0$,
$2\delta E(s) \leq \ell_\J(s)^2$. Therefore, for $s \geq s_0$,
\[-\dds E(s) \geq 2\delta  E(s).\]
Integrating, we conclude that
\begin{equation}
  \label{eq:decay}
  E(s) \leq E(s_0) e^{-2\delta (s-s_0)}.  
\end{equation}

We complete the proof using the mean value inequality for
pseudoholomorphic curves.
In the following, $c$ is a $u$-independent constant that changes across instances. 
Choose $0<r \ll 1$.  By the mean value
inequality \cite[Lemma 4.3.1 (ii)]{ms:jh}, for any $z=(s,t) \in C$,
$s > 2r +s_0$,
  \[|du(z)|_g^2 \leq \tfrac 8 {\pi r^2}\int_{B_{2r}(z) \cap C}|du|_g^2\,ds \,dt.\]
  By the uniform equivalence \eqref{eq:eqmet} of the $J$ and $g$ metrics, there is a constant $c>0$ such that for all $z=(s,t)$,
  \[|du(z)|_g \leq c \left(\int_{B_{2r}(z) \cap C}|du(s,t)|_{J_t}^2 \,ds \,dt\right)^{1/2}=c E(u,B_{2r}(z))^{1/2} \leq c E(s_0)^{1/2} e^{-\delta (s-s_0)}.\]
  Integrating we obtain
  \[d_\cyl(u(s,t),p) \leq \int_s^\infty |\partial_s u|_g ds \leq c E(s_0)^{1/2} e^{-\delta (s-s_0)}.\]
The proposition follows by taking $c_u:=cE(s_0)^{1/2}e^{\delta s_0}$. 
\end{proof}

\begin{remark} \label{rem:anydelta} To prove the decay estimate
  \eqref{eq:decayref} in Proposition \ref{prop:expdecay}, for some
  $\delta$, instead of a given $\delta$, one can use the 
  isoperimetric inequality (Proposition \ref{prop:iso}) in place of
  the refined version. In that case, it is enough to assume that the family
  $\{J_t\}_{t \in [0,1]}$ of almost complex structures in $\om$-tamed.
\end{remark}

\bibliography{rem}{}
\bibliographystyle{plain}

\end{document}